\documentclass[12pt]{amsart}
\usepackage{amsfonts, amssymb, amsmath, amsthm, euscript}
\usepackage[all]{xy}
\swapnumbers
\theoremstyle{plain}
\newtheorem{thm}{Theorem}[section]
\newtheorem{lem}[thm]{Lemma}

\theoremstyle{definition}

\newtheorem*{Ack}{Acknowledgement}

\theoremstyle{remark}

\def\Hom{\operatorname{Hom}}
\def\End{\operatorname{End}}
\def\lann{\operatorname{lann}}
\def\rann{\operatorname{rann}}

\def\dim{\operatorname{dim}}

\def\spec{\operatorname{Spec}}

\def\f{\frac}
\def\ds{\displaystyle}

\begin{document}

\title{Nonsplit module extensions over the one-sided inverse of $k[x]$}


\author{Zheping Lu}
\address{
  Tandon School of Engineering, New York University \\
  Brooklyn, NY 11201}
\email{zl2965@pitt.edu}

\author{Linhong Wang}
\address{Department of Mathematics\\
  University of Pittsburgh\\
  Pittsburgh, PA 15260}
\email{lhwang@pitt.edu}

\author{Xingting Wang}
\address{Department of Mathematics\\
  Howard University\\
  Washington, D.C. 20059}
\email{xingting.wang@howard.edu}



\begin{abstract}
Let $R$ be the associative $k$-algebra generated by two elements $x$ and $y$ with defining relation $yx=1$. A complete description of simple modules over $R$ is obtained by using the results of Irving and Gerritzen. We examine the short exact sequence $0\rightarrow U\rightarrow E \rightarrow V\rightarrow 0$, where $U$ and $V$ are simple $R$-modules. It shows that nonsplit extension only occurs when both $U$ and $V$ are one-dimensional, or, under certain condition, $U$ is infinite-dimensional and $V$ is one-dimensional.
\end{abstract}
\maketitle

\section{Introduction}
In this short note, we study nonsplit extensions of simple modules over the associative algebra $R=k\{x,y\}/\langle yx-1\rangle$ over a base field $k$ of characteristic 0. The algebra $R$ is also known as the one-sided inverse of the polynomial algebra $k[x]$ and appeared in the work of \cite{Bav}, \cite{Ger}, \cite{Jac}, and \cite{Irv}. Note that
\[y(1-xy)=(1-xy)x=0.\]
The algebra $R$ is not a domain, and $Z(R)=k$. As a $k$-vector space $R$ has basis
\[\big\{x^iy^j\, \mid\, i,j=0, 1, 2, \ldots \big\}.\]
Moreover, $R$ admits the involution $\eta: x\mapsto y$ and $y\mapsto x$. Hence, the left and right algebraic properties of $R$ are the same.

Jacobson \cite{Jac} gave a faithful irreducible representation of $R$ as follows. Let $S$ be the infinite-dimensional $k$-vector space with the basis $\{e_{1}, e_{2},\ldots\}$ and let $R$ act on $S$ by assigning the following:
\[
\begin{cases}
x\, e_{n}=e_{n+1}, & n>0,\\
y\,e_{n}=e_{n-1}, &  n>1,\\
y\, e_{1}=0.
\end{cases}
\]
It is proved by Bavula \cite{Bav} and Gerriten \cite{Ger} that there is only one isomorphic class of infinite-dimensional simple $R$-modules.
Note that there is an algebra monomorphism $R\to \End_k(k[x])$ such that $x\mapsto x$ and $y\mapsto H^{-1}\frac{d}{dx}$ where $H\in \End_k(k[x])$ is given by $H(f)=\frac{d}{dx}(xf)$ for any $f\in k[x]$. In particular,
\[\bigoplus_{i\ge 0} k\,x^i(1-xy)\cong k[x]\]
is a simple and faithful left $R$-module where the left $R$-module structure on $k[x]$ is via the algebra map $R\to \End_k(k[x])$ discussed above. Follows Bavula \cite{Bav}, $R$ contains a subring which is canonically isomorphic to the ring (without identity) of infinite dimensional matrices. Let
\[F=\bigoplus_{i,j\ge 0} k\,M_{ij} \cong M_{\infty}(k),\]
where $M_{ij}=x^i(1-xy)y^j$ can be identical to the matrix units of $M_{\infty}(k)$. In particular, we have
\[x \sim \left(%
\begin{array}{cccc}
 0 & \ &\ \ \ \ &\ \ \ \\
 1 & 0 & \ \ \ \ &\ \ \ \\
 \ & 1 & 0        &\ \ \ \\
 \ & \ & \ddots    &\ddots \\
\end{array}%
\right),\quad   \quad
y \sim \left(%
\begin{array}{cccc}
 0 & 1 &\  &\ \ \ \  \\
 \ & 0 & 1 &\ \ \ \  \\
 \ & \ & 0 & \ddots  \\
 \ & \ & \ &  \ddots  \\
\end{array}%
\right).\]
As a left $R$-module,
\[F=\bigoplus_{i,j\ge 0} k\,x^i(1-xy)y^j\cong \bigoplus_{i\ge 0} \left(\bigoplus_{t\ge 0}k\,x^t\,x^i(1-xy)y^i\right)\cong \bigoplus_{i\ge 0} k[x]\]
is a direct sum of infinitely many simple $R$-modules. Hence $R$ is neither left nor right noetherian. Similarly, we see that there is an ascending chain of left annihilators in $R$ which is not stable. Then $R$ is neither left nor right Goldie. Moreover, $F$ is equal to the ideal of $R$ generated by $\langle 1-xy \rangle$. Since $F^2=F$, $\lann(F)$ and $\rann(F)$ are both zero, we have $F$ is an essential left and right ideal of $R$, which equals the socle of left and right $R$-module $R$. Hence $F$ is contained in any nonzero ideal of $R$ and it follows that the set of proper (two-sided) ideals of $R$ is
\[\big\{0,\, \langle 1-xy \rangle,\, \langle 1-xy, f(x) \rangle \big\},\]
where $f(x)$ is a monic polynomial in $k[x]$ which is not a monomial. In particular, the ideals of $R$ satisfy the ascending chain condition.

It follows from the work of Bavula \cite{Bav},  Gerritzen\cite{Ger}, and Irving \cite{Irv} that the prime ideals are given by
\[\spec(R)=\left\{ 0,\, \langle 1-xy \rangle,\, \langle 1-xy, f(x) \rangle\right\},\]
where $f(x)$ is a monic irreducible polynomial in $k[x]$ which is not a monomial. In particular, $\langle 1-xy, f(x) \rangle$ are the maximal ideals of $R$. Therefore simple $R$-modules are isomorphic to $k[x]$ or $k[x^{\pm 1}]/ \langle f(x) \rangle$. When $k$ is algebraically closed, the simple $R$-modules are either 1-dimensional or infinite dimensional.
 
A discussion of how Jategaonkar's Main Lemma and a theorem of Stafford's apply to this non-noetherian $R$ is given in the close discussion section.  

\section{nonsplit extensions of simple $R$-modules}
Throughout $k$ is algebraically closed field with $\text{char}(k)=0$. All modules are left modules. Then simple $R$-modules are isomorphic to $k[x]$ or $k[x^{\pm 1}]/ \langle x-\lambda \rangle$ for $\lambda \in k^{\times}$. When a simple module is 1-dimensional, i.e., isomorphic to $k$ as a vector space, the $x$-action is the multiplication by a scalar $\lambda$, and the $y$-action is the multiplication by its inverse $\lambda^{-1}$. We denote such simple $R$-module by $k_{\lambda}$. It is clear that $k_{\lambda_1}\cong k_{\lambda_2}$ as simple $R$-modules for any $\lambda_1,\lambda_2\in k^\times$ if and only if $\lambda_1=\lambda_2$.

We consider the $R$-module extension $E$ with the short exact sequence (\emph{s.e.s.})
\begin{equation}
0\rightarrow U\rightarrow E \rightarrow V\rightarrow 0\tag{1}
\end{equation}
of $R$-modules $U$ and $V$. It is clear that $E$ is isomorphic to $U\oplus V$, as $k$-vector spaces. The $R$-action on $E$ is then given by the ring homomorphism
\[\rho_{\delta}: r\mapsto \left(%
\begin{array}{cc}
 \alpha (r) & \delta (r)\\
 0 & \beta(r)
\end{array}%
\right),\]
where
\[\alpha : R\longrightarrow \End_{k}(U)\quad \text{and}\quad  \beta : R\longrightarrow \End_{k}(V)\]
are ring homomorphisms, and $\delta(r)$ is a $k$-linear map in $\Hom_k(V, U)$ such that  \[\delta(r_{1}r_{2})=\alpha(r_1)\delta(r_2)+\delta(r_1)\beta(r_2)\] for any $r_{1},r_{2}\in R$. In particular, 
\begin{equation*}\label{com}
\alpha(y)\delta(x)+\delta(y)\beta(x)= \delta(yx)=\delta(1).
\end{equation*}
Since $\rho_{\delta}(1)$ must be the identity matrix, we have $\delta(1)=0$.
Therefore, \[\alpha(y)\delta(x)+\delta(y)\beta(x)=0.\] 
That is,  given $\alpha$ and $\beta$, the map $\delta$ is uniquely determined by the pair of $k$-linear maps $\delta(x),\delta(y)\in \Hom_k(V,U)$ satisfying the above compatible condition. 
If $\delta$ is the zero mapping, then $E\cong U\oplus V$. Let $E_{\delta}$ and $E_{\delta'}$ be two module extensions of $U$ by $V$, equipped with ring homomorphisms $\rho_{\delta}$ and $\rho_{\delta'}$. Then $E_{\delta} \cong E_{\delta'}$ if and only if there is a $k$-vector space isomorphism $f: E_{\delta} \longrightarrow E_{\delta'}$ such that $f\circ \rho_{\delta}(r)=\rho_{\delta'}(r)\circ f$. Note that $R$ has the $k$-basis $\{x^iy^j\, \mid\, i,j= 0, 1, 2, \ldots\}$. Therefore, it is sufficient to verify $\rho_{\delta}(x)=f^{-1}\circ \rho_{\delta'}(x)\circ f$ and $\rho_{\delta}(y)=f^{-1}\circ \rho_{\delta'}(y)\circ f$.

Now consider another $R$-module extension $E'$ with the short exact sequence (\emph{s.e.s.})
\begin{equation}
0\rightarrow U'\rightarrow E' \rightarrow V'\rightarrow 0 \tag{2}
\end{equation}
of $R$-modules $U'$ and $V'$. We say that the two \emph{s.e.s} (1) and (2) are \emph{equivalent} if there is a $R$-module isomorphism $f: E\to E'$ such that the restriction of $f$ on $U$ yields an isomorphism from $U$ to $U'$.

In this note, we focus on the $R$-module extension $E$ of a simple $R$-module $U$ by another simple $R$-module $V$. We start with the case when $V$ is infinite-dimensional. It is shown in the following lemma that the s.e.s in this case is always split. This result can be directly derived from Bavula's proof that the infinite-dimensional simple $R$-module $k[x]$ is projective. We include an alternative proof without using projectivity.

\begin{lem}
Suppose $0\rightarrow U \rightarrow E_{\delta} \rightarrow V \rightarrow 0$ is a s.e.s. , where $U$ and $V$ are simple $R$-modules and $\dim_k(V)=\infty$. Then the s.e.s. is always split. 
\end{lem}
\begin{proof}
Let $\{b_0, b_1, b_2, \ldots\}$ be a basis of $V$ such that $y$ and $x$ are left and right shift operators, respectively. As vector spaces, $E_\delta\cong U\oplus V$. Consider the element 
\[a:=b_0-x \delta(y) b_0\]
 of $E_\delta$. It is clear that $a \in E_\delta\setminus U $. Then the left cyclic submodule $Ra$ of $E_\delta$ is distinct from 0 and $U$. For any element $r\in R$, we have 
 \[ra=\delta(r)b_0+\beta(r)b_0-rx\delta(y)b_0.\]
Hence $ra\in R_a \cap U$ only if $\beta(r)b_0=0$, that is, $r=sy$ for some $s\in R$. But 
\[ya=yb_0-yx\delta(y)b_0=\delta(y)b_0+\beta(y)b_0-\delta(y)b_0=0.\]
That is, $R_a\cap U=0$. Then $R_a \oplus U=E_\delta$ since $E_\delta/U\cong V$ is simple. Therefore $E_\delta \cong U \oplus V$ as left $R$-modules.
\end{proof}

The next case deals with the module extension when $U$ is infinite-dimensional and $V$ is one-dimensional.

\begin{lem}
Let $U$ and $U'$ be two infinite-dimensional simple $R$-modules,  $k_{\lambda} $ and $k_{\lambda'}$ be two one-dimensional $R$-modules for nonzero scalars $\lambda$ and $\lambda'$. Suppose $E_{\delta}$ and $E_{\delta'}$ are two $R$-module extensions with the  s.e.s. 
\[ 0\rightarrow U \rightarrow E_{\delta} \rightarrow k_{\lambda} \rightarrow 0 \quad \text{and} \quad 0\rightarrow U' \rightarrow E_{\delta'} \rightarrow k_{\lambda'} \rightarrow 0,\quad  resp.\] Then $E_{\delta} \cong E_{\delta'}$ if and only if $\lambda=\lambda'$ and $\delta'(x)=c\delta(x)$ for some nonzero $c\in k$. In this case the two s.e.s. are equivalent if and only if $E_{\delta} \cong E_{\delta'}$. As a consequence, $E_\delta$ (resp. $E_{\delta'}$) is nonsplit if and only if $\delta\neq 0$ (resp. $\delta'\neq 0$).
\end{lem}

\begin{proof}
We will fix a basis $\{e_0, e_1, e_2, \ldots, d \}$ for both $E_{\delta}$ and $E_{\delta'}$ as $k$-vector spaces, where $\{e_0, e_1, e_2, \ldots \}$ is a basis of $U$ (and $U'$) such that $y$ and $x$ are left and right shift operators, respectively. For any $r\in R$, we can identify the map $\delta(r)$, under the fixed basis, with an infinite-dimensional vector
\[\langle \delta(r)_0,\, \delta(r)_1,\, \delta(r)_2,\, \ldots\rangle\]
with only finitely many nonzero components. Note that $\alpha(y)\delta(x)+\delta(y)\beta(x)=0$, where $\beta(x)=\lambda$ and that $\alpha (y)$ is the upper diagonal line matrix given in section 1. It follows that 
\begin{align*}
\delta(y)_i=\lambda^{-1} \delta(x)_{i+1}\quad  \text{for} \quad i\geq 1\tag{1}
\end{align*}
Similar result for $\delta'(x)$ and $\delta'(y)$ holds. Suppose that $m$ is the smallest integer such that $\delta(y)_i=\delta'(y)_i=0$ for any $i >m$.  Consequently, $\delta(x)_{i}=\delta'(x)_{i}=0$ for any $i>m+1$.

Suppose that $f$ is a $R$-module isomorphism $E_{\delta'} \to E_{\delta}$, that is, $f$ is a $k$-vector space isomorphism such that both $\rho_{\delta}(x)f=f\rho_{\delta'}(x)$ and $\rho_{\delta}(y)f=f\rho_{\delta'}(y)$. We will obtain necessary conditions on $f$ through its images on the basis elements of the selected basis.
Let 
\[f(e_0)=ae_0+a_1e_1+a_2e_2+\cdots
+a'd,\] 
for some $a', a_{i}\in k, i=1, 2, \ldots$, where only finitely many $a_{i}$'s are nonzero. Firstly, 
\begin{align*}
f\circ \rho_{\delta'}(y)(e_0)&=0, \quad \text{and}\\
\rho_{\delta}(y)\circ f(e_0)&=\sum_{i\ge 0}\, (a_{i+1}+a'\delta(y)_i)e_i+ \f{1}{\lambda}a'd.
\end{align*}
Hence, $a'= a_{i}=0$ for all $i=1, 2, \ldots$, and so $f(e_0)=ae_0$. 
Moreover, \[f(e_1)=f(xe_0)=xf(e_0)=x(ae_0)=ae_1\] implies that $f(e_1)=ae_1$. Inductively, $f(e_i)=ae_i$ for some $a \neq 0$ and all $i\geq 0$.
Next,  suppose that 
\[f(d)=b_0e_0+b_1e_1+b_2e_2+\cdots
+bd,\]
where $b\neq 0$, $b_{i}\in k$ for $i\geq 0$ and only finitely many $b_{i}$'s are nonzero.  Then
\begin{align*}
\rho_{\delta}(y)\circ f(d)&= \sum_{i\ge 0} b_{i+1}e_i+ \sum_{i\ge 0}b\delta(y)_i e_i + \lambda^{-1}bd,
\\
f\circ \rho_{\delta'}(y)\, (d)&= \sum_{i\ge 0}\Big(a\delta'(y)_i + \f{1}{\lambda'}b_i\Big) e_i +\f{1}{\lambda'}bd.
\end{align*}
Thus, we have
\[\lambda=\lambda', \quad b_{i+1}+ b\delta(y)_i=a\delta'(y)_i + \lambda^{-1}b_i\quad \text{for}\quad i\geq 0.\]
Since $\delta(y)_i=\delta'(y)_i=0$ for any $i >m$, we have $b_{i+1}=\lambda^{-1} b_i$ for any $i >m$. 
But only finitely many $b_{i}$'s are nonzero, it then follows inductively that 
\[b_{m+1}=b_{m+2}=\ldots=0.\] 
Hence, we have the $m+1$ relations
\begin{align*}
\begin{cases}
 \hspace*{.4in} b\delta(y)_m&=\quad a\delta'(y)_m + \lambda^{-1}b_m,\\
 b_{i+1}+ b\delta(y)_i&=\quad a\delta'(y)_i + \lambda^{-1}b_i \quad \text{for}\quad i=0, 1,  \ldots, m-1.
\tag{2}
\end{cases} 
\end{align*}
Similarly, we have
\begin{align*}
\rho_{\delta}(x)\circ f(d)&= \sum_{i\ge 1} b_{i-1}e_{i}+ \sum_{i\ge 0}b\delta(x)_i e_i + \lambda bd,
\\
f\circ \rho_{\delta'}(x)\, (d)&=\sum_{i\ge 0}\Big( a\delta'(x)_i + \lambda'b_i \Big) e_i + \lambda' bd.
\end{align*}
Note that $\delta(x)_j=\delta'(x)_j=0$ for any $j>m+1$.  It then follows that
\begin{align*}
\begin{cases}
 \hspace*{.4in}  b\delta(x)_0 &=\quad a\delta'(x)_0 + \lambda b_0, \\
b_m + b\delta(x)_{m+1} &=\quad a\delta'(x)_{m+1},\\
b_{i-1} + b\delta(x)_{i} &=\quad a\delta'(x)_{i} +\lambda b_{i} \quad \text{for}\quad i=1, 2, \ldots, m.  \tag{3}
\end{cases}
\end{align*}
Combining  the relations (1) and (3), we have
\begin{align*}
\begin{cases}
b\delta(y)_m-a\delta'(y)_m&=\quad -\lambda^{-1}b_m,\\
b\delta(y)_i -a\delta'(y)_i &=\quad b_{i+1}- \lambda^{-1}b_{i} \quad \text{for}\quad i=0, 1, \ldots, m-1. 
\end{cases}
\end{align*}
From (2), we have 
\begin{align*}
\begin{cases}
b\delta(y)_m-a\delta'(y)_m &=\quad \lambda^{-1}b_m,\\
b\delta(y)_i -a\delta'(y)_i &=\quad \lambda^{-1}b_{i} -b_{i+1} \quad \text{for}\quad i=0, 1, \ldots, m-1. 
\end{cases}
\end{align*}
Hence, $b_i=\lambda b_{i+1}$ for $0\leq i\leq m-1$ and $b_m=0$. Thus, $b_0=b_1=\ldots=b_m=0$. 

Therefore, $f(e_i)=ae_i$ and $f(d)=bd$ for some nonzero scalars $a,\, b\in k$ and all $i\geq 0$. Such a $k$-vector space isomorphism is a $R$-module isomorphism if and only if $\ds{\delta'(x)=\f{b}{a}\delta (x)}$ for the nonzero scalars $a, b\in k$ or equivalently, $\ds{\delta'(r)=\f{b}{a}\delta (r)}$ for any $r\in R$.

Therefore, any module extension $E_{\delta}$ such that $E_{\delta}/U\cong k_{\lambda}$ is nonsplit if and only if $\delta(x)\neq0$. Let $E_{\delta}$ and $E_{\delta'}$ be nonsplit extensions such that 
\[ 0\rightarrow U \rightarrow E_{\delta} \rightarrow k_{\lambda} \rightarrow 0 \quad \text{and} \quad 0\rightarrow U' \rightarrow E_{\delta'} \rightarrow k_{\lambda'} \rightarrow 0.\]
Then $E_{\delta} \cong E_{\delta'}$ if and only if $\lambda =\lambda'$ and $\delta'(x)=c \delta (x)$ for some nonzero scalar $c\in k$. Observe that the isomorphism $f$ from $E_{\delta}$ to $E_{\delta'}$ yields an isomorphism from $U$ to $U'$. Therefore, the two \emph{s.e.s.} are equivalent if and only if $E_{\delta} \cong E_{\delta'}$. 
\end{proof}

Now we can state our main result. 

\begin{thm}
Suppose $0\rightarrow U \rightarrow E_{\delta} \rightarrow V \rightarrow 0$ is a s.e.s.  where $U$ and $V$ are simple $R$-modules and $E_{\delta}$ is associated with the $k$-linear map $\delta$ in $\Hom_k(V, U)$. Let $\lambda$, $\lambda'$ be nonzero scalars. 
\begin{itemize}
\item[i)] If $\dim(V)=\infty$, the s.e.s. is always split. 
\item[ii)] If $\dim(U)=\infty$ and $V=k_{\lambda}$,  the s.e.s.  is nonsplit if and only if $\delta\neq 0$. Any such two s.e.s. are  equivalent if and only if $\lambda=\lambda'$ and the infinite vectors $\delta(x)$ and $\delta'(x)$ are proportional.
\item[iii)] If $U=k_{\lambda}$ and $V=k_{\lambda'}$ are both one-dimensional, then the s.e.s.  is nonsplit only if $\delta \neq 0$ and $\lambda=\lambda'$. Any such two nonsplit s.e.s. are equivalent if and only if the submodules $U$ are the same.
\end{itemize}

\end{thm}
\begin{proof}
The first two cases are proved in Lemma 2.1 and Lemma 2.2. We only need to consider the case when $U$ and $V$ are both one-dimensional. Suppose the two modules $U$ and $V$ are uniquely determined by nonzero scalars $\lambda$ and $\lambda'$. Let 
$$0\rightarrow k_{\lambda} \rightarrow E_{\delta}\rightarrow k_{\lambda'}\rightarrow 0$$ be a \emph{s.e.s.} Then $\delta$ is uniquely determined by $\delta(x)$ since $\delta(y)=-(\lambda \lambda')^{-1}\delta(x)$. Moreover, $\rho_{\delta}(y)$ is the inverse matrix of $\rho_{\delta}(x)$. 
Note that the $2\times 2$ matrix $\rho_{\delta}(x)$ is similar to $\rho_0(x)$ if and only if $\lambda \neq \lambda'$. Hence, the \emph{s.e.s.} is always split if $\lambda \neq \lambda'$, no matter $\delta=0$ or not. Therefore, the nonsplit case occurs when $\delta \neq 0$ and $\lambda=\lambda'$. Consider two nonsplit  \emph{s.e.s.} 
\[ 0\rightarrow k_{\lambda} \rightarrow E_{\delta} \rightarrow k_{\lambda} \rightarrow 0 \quad \text{and} \quad 0\rightarrow k_{\gamma} \rightarrow E_{\delta'} \rightarrow k_{\gamma} \rightarrow 0,\]
with nonzero $\delta$ and $\delta'$. It is easy to see, by a linear transformation, that the two  nonsplit  \emph{s.e.s.}  are equivalent if and only if  $E_{\delta}\cong E_{\delta'}$ if and only if the nonzero scalars $\lambda=\gamma$. Thus, there is only one, up to equivalence, nonsplit \emph{s.e.s.}  $0\rightarrow k_{\lambda} \rightarrow E_{\delta} \rightarrow k_{\lambda} \rightarrow 0$ for each  one-dimensional simple $R$-module $k_{\lambda}$.
\end{proof}

\section{Close discussion}
Let $A$ be an associative ring. Recall a left (respectively, right) module $M$ over $A$ is called \emph{torsionfree} if for any nonzero element $m$ in $M$ there is some $r\in A$ such that $rm\neq 0$ (respectively, $mr\neq 0$). Two prime ideals $P$ and $Q$ of an associative ring $A$ are \emph{linked}, denoted as $P\rightsquigarrow Q$, if there is an ideal $I$ of $A$ such that $(P\cap Q)>I\ge PQ$ and $(P\cap Q)/I$ is nonzero and torsionfree both as a left $A/P$-module and a right $A/Q$-module. The graph of links of $A$ is a directed graph whose vertices are prime ideals of $A$, with an arrow from $P$ to $Q$ whenever $P\rightsquigarrow Q$. The vertex set of each connected component is called a \emph{clique}.

Jategaonkar's Main Lemma \cite{Jate} states that if $M$ is a (right) module over a noetherian ring $A$ with a nonsplit short exact sequence $0\rightarrow U\rightarrow M \rightarrow V\rightarrow 0$ and corresponding annihilators  $Q=\text{ann}_{A}(U)$ and $P=\text{ann}_{A}(V)$, then exactly one of the following two alternatives occurs i) $P < Q$ and $PM=0$; ii) $P \rightsquigarrow Q$. 

Now let $0\rightarrow U\rightarrow E_{\delta} \rightarrow V\rightarrow 0$ be a nonsplit short exact sequence, where $U$ and $V$ are simple $R$-modules. Suppose $Q=\text{ann}_{R}(U)$ and $P=\text{ann}_{R}(V)$ are the affiliated primes.
When $\dim U=\infty$ and $V\cong k_{\lambda}$, we have $Q=(0)$ and $P=\langle 1-xy,\, x-\lambda \rangle$.
There is no link between $P$ and $Q$, neither $P < Q$. When $U\cong V\cong k_{\lambda}$, we have $Q=P=\langle 1-xy,\, x-\lambda \rangle$. There is no link between $P$ and $Q$, neither $P < Q$. This suggests that the noetherianess is necessary in the assumptions of Jategaonkar's Main Lemma.

On the other hand, a theorem of Stafford \cite[Corollary 3.13]{Stafford} states that all cliques of prime ideals in any noetherian ring are countable. When $k$ is algebraically closed, the prime ideals of $R$ are $(0)$, $F=\langle 1-xy\rangle$, and $P_\lambda=\langle 1-xy,\, x-\lambda\rangle$, where $\lambda\in k^\times$. One can check that 
\[
F=F^2=F\cap P_\lambda=FP_\lambda=P_\lambda F=P_\lambda\cap P_{\lambda '}=P_\lambda P_{\lambda'}
\]
whenever $\lambda\neq \lambda'$. Moreover, 
$
P_\lambda/P_{\lambda}^2\cong (x-\lambda)/(x-\lambda)^2
$
as in $k[x^{\pm 1}]$. Hence the cliques in the graph of links are
\[\xymatrix{
F && (0) && P_\lambda \ar@{->}@(ul,ur)  && P_{\lambda'} \ar@{->}@(ul,ur)\\
}
\]
This suggests that all cliques of $R$ are countable.

\begin{Ack}
The second author would like to express her gratitude to Professor E. Letzter for his valuable suggestions. The first author was a math major at the University of Pittsburgh who participated in an undergraduate research project that was related to this short article.  We are thankful to the math department of the University of Pittsburgh for its supporting. We also would like to express our gratitude to the referee for carful reading and helpful suggestions.
\end{Ack}

\end{document}